\documentclass[12pt]{amsart}
\usepackage{enumerate}
\usepackage[nobysame]{amsrefs}
\usepackage{enumitem}
\usepackage{amssymb}
\usepackage{amsmath}
\usepackage{amsthm}
\usepackage{microtype}

\usepackage{fullpage}

\theoremstyle{plain}
\newtheorem{theorem}{Theorem}[section]
\newtheorem{corollary}[theorem]{Corollary}
\newtheorem{lemma}[theorem]{Lemma}

\theoremstyle{definition}
\newtheorem{remark}[theorem]{Remark}
\newtheorem{example}[theorem]{Example}

\newtheorem{definition}[theorem]{Definition}
\newtheorem*{notation}{Notation}

\numberwithin{equation}{section}

\DeclareMathOperator*{\var}{var}

\DeclareMathOperator*{\osc}{osc}

\newcommand{\norm}[1]{\left\|#1\right\|}

\newcommand{\abs}[1]{\lvert#1\rvert}
\newcommand{\babs}[1]{\bigl|#1\bigr|}
\newcommand{\bgabs}[1]{\biggl|#1\biggr|}
\newcommand{\bnorm}[1]{\bigl\|#1\bigr\|}
\newcommand{\bgnorm}[1]{\biggl\|#1\biggr\|}

\newcommand{\bset}[2]{\bigl\{#1:#2\bigr\}}
\newcommand{\bgset}[2]{\biggl\{#1:#2\biggr\}}

\newcommand{\absp}[1]{\left.\!\!\left\bracevert\!\! \vphantom{Iy}#1\!\! \right\bracevert\!\!\right.}
\newcommand{\zdef}{{\mathrel{\mathop:}=}}
\newcommand{\defz}{{=\mathrel{\mathop:}}}

\author[D.~Bugajewska]{Daria Bugajewska}
\address{Daria Bugajewska, Optimization and Control Theory Department,
Faculty of Mathematics and Computer Science,
Adam Mickiewicz University,
ul.\ Umultowska 87,
61\nobreakdash-614 Pozna\'n,
Poland}
\email{dbw@amu.edu.pl}

\author[G.~Infante]{Gennaro Infante}
\address{Gennaro Infante, Dipartimento di Matematica e Informatica,
Universit\'a della Calabria,
87036 Arcavacata di Rende,
Cosenza,
Italy}
\email{gennaro.infante@unical.it}

\author[P.~Kasprzak]{Piotr Kasprzak}
\address{Piotr Kasprzak, Optimization and Control Theory Department,
Faculty of Mathematics and Computer Science,
Adam Mickiewicz University,
ul.\ Umultowska 87,
61\nobreakdash-614 Pozna\'n,
Poland}
\email{kasp@amu.edu.pl}

\title[Solvability of integral equations with applications to BVPs]{Solvability of Hammerstein integral equations\\ with applications to boundary value problems}


\begin{document}

\begin{abstract}
In this paper we present some new results regarding the solvability of nonlinear Hammerstein integral equations in a special cone of continuous functions. The proofs are based on a certain fixed point theorem of Leggett and Williams type. We give an application of the abstract result to prove the existence of nontrivial solutions of a periodic boundary value problem. We also investigate, via a version of Krasnosel{\cprime}slki{\u\i}'s theorem for the sum of two operators, the solvability of perturbed Hammerstein integral equations in the space of continuous functions of bounded variation in the sense of Jordan. As an application of these results, we study the solvability of a boundary value problem subject to integral boundary conditions of Riemann--Stieltjes type. Some examples are presented in order to illustrate the obtained results.
\end{abstract}

\maketitle

\section{Introduction}
Numerous problems of various branches of science lead to the necessity of investigating the solvability of nonlinear Hammerstein integral equations. It is for this reason the theory of nonlinear integral equations has become an important part of nonlinear functional analysis and has attracted the interest of many mathematicians.

In the first part of the paper we study the existence of eigenvalues of Hammerstein integral equations of the form
\begin{equation}\label{eqh-intro}
\lambda x(t)=\int_{\overline{\Omega}} k(t,s)f(s,x(s)) \textup ds,\quad t\in\overline{\Omega},
\end{equation}
where $\Omega$ is an open and bounded subset of $\mathbb R^n$, $k$ is allowed to change sign and $f$ is non-negative. 

A number of tools have been utilized to study the solvability of~\eqref{eqh-intro}; for example, variational methods have been employed in the case of \emph{symmetric} kernels by Faraci~\cites{F} and Faraci and Moroz~\cite{FM}, topological methods have been used by 
Lan~\cites{KL1,KL2,KL3}, Lan and Webb~\cite{LWe}  and Ma~\cite{Ma00}. In particular, Lan~\cite{KL1} proved the existence of a positive eigenfunction under non-negativity assumptions on the kernel $k$. The results of~\cite{KL1} were complemented by Infante~\cite{giems}, who proved, under weaker assumptions regarding  the sign of the kernel $k$, the existence of eigenfunctions within a cone of functions that are allowed to \emph{change sign}, namely 
\begin{equation}\label{cone-chs}
K= \bgset{x \in C(\overline{\Omega})}{ \min_{t \in \Omega_0} x(t) \geq c\norm{x}_{\infty}},
\end{equation}
where $\Omega_0\subseteq\overline{\Omega}$ is a closed set of positive Lebesgue measure. This type of cone has been introduced by Infante and Webb in~\cite{I-W1}. Let us note that the functions in~\eqref{cone-chs} are positive on the subset $\Omega_0$ but are allowed to change sign elsewhere.

A key assumption in~\cite{giems} is the positivity of the kernel $k$ on $\Omega_0 \times \overline{\Omega}$. Here we drop this condition and utilize instead the cone
\begin{equation}\label{cone-posevg}
 C= \bgset{x \in C(\overline{\Omega})}{\int_{\Omega_0} x(t) \textup dt \geq c\norm{x}_{\infty}}.
\end{equation}
Let us note that we do not require  the functions in $C$ to be \emph{positive} in $\Omega_0$, but only to have a  \emph{positive average} in this subset. In Section~\ref{sec:LW} we use a recent Leggett--Williams type theorem due to Bugajewski and Kasprzak~\cite{BK} in order to prove the existence of eigenfunctions for~\eqref{eqh-intro} in the cone~\eqref{cone-posevg}. Let us note that we do not require symmetry for our kernels.

Motivated by the works of Graef, Kong and Wang~\cite{G} and Webb~\cite{W2}, we apply our results to study the periodic boundary value problem (BVP)
\begin{gather}
x''(t)+\omega^2 x(t)=\lambda f(t,x(t)), \quad t \in [0,1], \label{eq:1}\\
x(0)=x(1), \quad x'(0)=x'(1). \label{eq:2}
\end{gather}

Our approach, in comparison with the ones used in the papers~\cites{G, W2}, has the advantage of working for all admissible values of the parameter $\omega$, that is, for all $\omega >0$ such that $\omega \neq 2n\pi$, $n \in \mathbb N$ (in other words, we are able to prove the existence of solutions to the periodic BVP~\eqref{eq:1}--\eqref{eq:2} even in the case when the corresponding Green's function takes negative values -- see Theorem~\ref{thm:existence_LW}). It also enables us to pinpoint the localization of the solution by means of the supremum and the integral norm. Let us add that if $\omega \in (0,\pi]$ our result ensures the existence of positive solutions to~\eqref{eq:1}--\eqref{eq:2} for some $\lambda>0$.

Note that, although we impose strictly weaker conditions on the function $f$ than the authors of~\cites{G,W2}, our result -- Theorem~\ref{thm:existence_LW} -- is in a sense `incomparable' with the existence results obtained in~\cites{G,W2} (for the discussion of the relation of Theorem~\ref{thm:existence_LW} and the theorems from~\cites{G,W2} we refer the reader to Remark~\ref{rem:inc} in Section~\ref{sec:BVP_LW}).

In the second part of the paper (see Section~\ref{sec:Kras}), by means of some version of Krasnosel{\cprime}slki{\u\i}'s theorem for the sum of two operators, we study the existence of solutions of the following perturbed nonlinear Hammerstein integral equation
\begin{equation}\label{eq:bvp_hammerstein-intro}
 x(t)=\alpha[x]v(t) + \beta[x]w(t) + \lambda \int_0^1 k(t,s)f(s,x(s))\textup ds, \quad t \in [0,1],
\end{equation}
in the space $CBV[0,1]$ consisting of continuous functions of bounded variation in the sense of Jordan; here $\lambda \in \mathbb R$, $\alpha, \beta \in CBV^*[0,1]$ and $v,w \in CBV[0,1]$. Let us add that the approach via Krasnosel{\cprime}slki{\u\i}'s fixed point theorem relies heavily on the ideas and techniques concerning the compactness of nonlinear integral operators in the space of functions of bounded variation developed in~\cite{BGK}. (It may seem surprising, but sufficient conditions guaranteeing the compactness of Hammerstein integral operators in the space $BV[0,1]$ of functions of bounded variation, were not described until very recently -- for more details see~\cite{BGK}.)

The motivation to seek solutions of bounded variation is connected with their numerous applications (see e.g. \cites{AFP,Maz,Wa}). 
Such functions can be used to describe some real world phenomena; for example, functions of bounded variation appear in mathematical biology or economics (see~\cites{Br, HL}). Furthermore, it turns out that by
a suitable choice of the space of functions of bounded variation it is possible, for example, to obtain solutions to certain nonlinear Hammerstein integral equations which are constant on each interval of continuity (for more details see \cite{BBL}).

The necessity of studying the integral equation \eqref{eq:bvp_hammerstein-intro} in the space $CBV[0,1]$ comes, \emph{for example}, from the fact that it naturally arises, when dealing with BVPs 
\begin{equation}
x''(t)=-\lambda f(t,x(t)), \qquad t \in [0,1],\label{eq:3}
\end{equation}
with non-local boundary conditions (BCs) of the form
\begin{equation}\label{eq:4}
 x(0)=\int_0^1 A(s)\textup dx(s) \quad \text{and} \quad x(1)=\int_0^1 B(s)\textup dx(s),
\end{equation}
or
\begin{equation}\label{eq:4'}
x(0)=\int_0^1 x(s)\textup dA(s) \quad \text{and} \quad x(1)=\int_0^1 x(s)\textup dB(s);
\end{equation}
let us add that the integrals occurring in~\eqref{eq:4} and~\eqref{eq:4'} are understood in the Riemann--Stieltjes sense.

In the context of ODEs the study of non-local BCs can be traced back to Picone~\cite{Picone}, who investigated multi-point BCs. For an introduction to non-local problems we refer to the reviews~\cites{Conti, rma, sotiris, Stik, Whyburn} as well as  the
papers~\cites{kttmna, ktejde, I-W3, Bsz}.

The existence of positive solutions (and their multiplicity) to the BVP \eqref{eq:3}--\eqref{eq:4'} was studied, for example, in \cites{I2009, W1, I-W2, I-W3}. In this paper (see Section~\ref{ssec:BVP_nl}), using a different approach from the one used in the above-mentioned articles, we establish existence type results for the BVPs with non-local integral conditions \eqref{eq:3}--\eqref{eq:4} and \eqref{eq:3}--\eqref{eq:4'}. Furthermore, we provide some examples of multi-point BVPs to which our results apply (see Examples~\ref{ex:111} and~\ref{ex:3}).  

\section{Preliminaries}
The aim of this Section is to fix the notation and to recall some basic definitions and facts which will be used in the sequel.

\begin{notation}
The closed ball in a normed space $X$ with center at $x$ and radius $r \in (0,+\infty)$ will be denoted by $B_X(x,r)$. For simplicity, instead of $B_{\mathbb R}(x,r)$ we will simply write $[x-r,x+r]$. The symbol $\theta$ will stand for the zero element of the normed space $X$.

If $\Omega$ is an open and bounded subset of $\mathbb R^n$, then by $C(\overline{\Omega})$ we will denote the Banach space of all continuous real-valued functions defined on $\overline{\Omega}$, endowed with the supremum norm $\norm{\cdot}_{\infty}$. In the particular case when $\Omega$ is an open interval, that is $\Omega=(a,b)$, we will write $C[a,b]$ instead of $C(\overline{(a,b)})$ or $C([a,b])$. Moreover, by $BV[a,b]$ we will denote the Banach space of all real-valued functions defined on $[a,b]$ of bounded variation (in the sense of Jordan), endowed with the norm $\norm{f}_{BV}\zdef \abs{f(a)} + \var_{[a,b]} f$, and by $CBV[a,b]$ its closed subspace consisting of continuous functions; here the symbol $\var_{[a,b]} f$ denotes the (Jordan) variation of the function $f \colon [a,b] \to \mathbb R$, that is, 
\[
\var_{[a,b]} f=\sup_{\pi}\sum\limits_{i=1}^n\abs{f(t_i)-f(t_{i-1})},
\]
where the supremum is taken over all finite partitions $\pi : a=t_0<t_1<\ldots<t_n=b$ of the interval $[a,b]$. If no confusion concerning the interval over which we compute the variation can arise, instead of $\var_{[a,b]} f$ we will simply write $\var f$. Let us also recall that $BV$-functions are bounded and $\norm{f}_{\infty} \leq \norm{f}_{BV}$ for every $f \in BV[0,1]$. As usual, by $CBV^{\ast}[a,b]$ we will denote the dual space of $CBV[a,b]$, that is, the space of all continuous linear functionals $\alpha \colon CBV[a,b] \to \mathbb R$. For a thorough treatment of functions of bounded variation of various kinds we refer the reader to~\cite{ABM}. 

Although throughout the paper we will use the same symbol `$\int$' to denote both the Lebesgue and the Riemann--Stieltjes integral, it should always be clear from the context which integral we use. The Lebesgue measure in $\mathbb R^n$ will be denoted by $\mu$.
\end{notation}

\subsection{Partially ordered structures}

Now, we are going to recall some definitions concerning partially ordered structures.

\begin{definition}[cf.~\cite{LW}*{p.~249}]\label{def:cone}
Let $X$ be a real normed space. A non-empty closed and convex set $C_X \subseteq X$ is called a~(positive) cone if the following conditions are satisfied:
\begin{enumerate}[label=(\alph*)]
 \item\label{it:cone_1} if $x \in C_X$ and $\lambda \geq 0$, then $\lambda x \in C_X$;
 \item\label{it:cone_2} if $x \in C_X$ and $-x \in C_X$, then $x=\theta$.
\end{enumerate}
\end{definition}

\begin{notation}
If $C_X$ is a cone in a normed space $X$, then by $C_X(\theta,r)$ we will denote the intersection of $C_X$ and $B_X(\theta,r)$, that is, $C_X(\theta,r)\zdef C_X \cap B_X(\theta,r)$.
\end{notation}

\begin{remark}
A cone $C_X$ in a normed space $X$ induces a partial order $\preceq$ given by the following formula
\begin{equation}\label{porzadek}
 x \preceq y \qquad \text{if and only if} \qquad y-x \in C_X.
\end{equation}
Let us note that the relation $\preceq$ is compatible with the linear structure of the normed space $X$, that is, if $x \preceq y$, then $x+z \preceq y+z$ and $\lambda x\preceq \lambda y$ for all $x,y,z \in X$ and all $\lambda \geq 0$.
\end{remark}

In the sequel, considering an ordered normed space with a cone, we will always assume that the partial order is defined by the formula~\eqref{porzadek}. 

\begin{definition}[cf.~\cite{LW}*{p.~249}]
A real Banach space  endowed with the partial order induced by a cone is called an \emph{ordered Banach space}.
\end{definition}

\subsection{Fixed point theorems}

Let us begin with recalling the following  extension of the Leggett--Williams theorem.

\begin{theorem}[\cite{BK}*{Theorem~4}]\label{thm:LW2}
Let $(X, \norm{\cdot})$ be an ordered Banach space with a cone $C_X$ and suppose $F \colon C_X(\theta,r) \to C_X$, where $r>0$, is a compact mapping. Moreover, assume that a continuous seminorm  $\absp{\vphantom{I}\cdot} \colon X \to [0,+\infty)$, together with positive numbers $m,M$ and $\delta$ satisfy the conditions\textup:
\begin{enumerate}[label=\textup{(\roman*)}, itemsep=2mm]
 \item\label{it:ar_2} $\norm{x}\leq M \absp{\vphantom{I}x}$ if $x \in C_X$\textup;
 \item\label{it:ar_3} $m \leq rM^{-1}$ and $\absp{\vphantom{I}x}=m$ for some $x \in C_X(\theta,r)$\textup;
 \item\label{it:LW2_4} $\absp{F(x)}\geq \delta$ if $x \in C_X(\theta,r)$ and $\absp{\vphantom{I}x}=m$. 
\end{enumerate}
Then there exist $\lambda_0 >0$ and $x_0 \in C_X(\theta,r) \setminus \{\theta\}$ such that $F(x_0)=\lambda_0 x_0$ and $\absp{\vphantom{I}x_0}=m$.
\end{theorem}

In the sequel, we will also use the following version of the well-known Krasnosel{\cprime}slki{\u\i}'s fixed-point theorem for the sum of two operators. In comparison with the original result, the nonlinear contraction has been replaced by a bounded linear operator with spectral radius less than one (for a quite general version of Krasnosel{\cprime}slki{\u\i}'s fixed-point theorem encompassing almost 30 previously known generalizations of that theorem, we refer the reader to~\cite{Park}). For completeness, let us recall that the spectral radius of a bounded linear operator $A$ is given by the formula $r(A):=\lim_{n \to \infty}\norm{A^n}^{\frac{1}{n}}$ (see, for example,~\cite{Leb}*{p.~109}).

\begin{theorem}\label{thm:krasnosielskii}
Let $M$ be a non-empty closed and convex subset of a Banach space $X$, and let $F_1\colon X \to X$, $F_2\colon M \to X$ be two mappings such that\textup:
\begin{enumerate}[label=\textup{(\roman*)}]
 \item\label{it:krasnosielski_2} $F_1$ is a bounded linear operator such that $r(F_1)<1$\textup;

 \item\label{it:krasnosielski_1} $F_2$ is compact, that is, $F_2$ is continuous and $F_2(M)$ is contained in a compact subset of $X$\textup;

 \item\label{it:krasnosielski_3} if $x=F_1(x)+F_2(y)$ for some $y \in M$, then $x\in M$. 
\end{enumerate}
Then $F_1+F_2$ has a fixed point in $M$.
\end{theorem}

The proof of Theorem~\ref{thm:krasnosielskii} is standard, and hence it will be omitted.

\begin{remark}\label{rem:krasosielski}
The condition~\ref{it:krasnosielski_3} of Theorem~\ref{thm:krasnosielskii}, which guarantees that a certain nonlinear operator maps the set $M$ into itself, is often the most difficult condition among the three one has to verify. However, if $M:=B_X(\theta,r)$ and the assumption~\ref{it:krasnosielski_1} of Theorem~\ref{thm:krasnosielskii} is strengthened to: `$F_2$ is continuous and $F_2(M)$ is contained in a compact subset of $B_X\bigl(\theta,r/\norm{(I-F_1)^{-1}}\bigr)$,' where $\norm{(I-F_1)^{-1}}$ denotes the  operator norm of the continuous inverse of $I-F_1$ which exists due to~\ref{it:krasnosielski_2}, then the condition~\ref{it:krasnosielski_3} of Theorem~\ref{thm:krasnosielskii} holds. Indeed, if $x=F_1(x)+F_2(y)$, then $(I-F_1)(x)=F_2(y)$, and thus $\norm{x} = \norm{(I-F_1)^{-1} \circ F_2(y)} \leq \norm{(I-F_1)^{-1}}\cdot \norm{F_2(y)} \leq r$. Hence, $x\in M$.
\end{remark}

\begin{remark}
Let us also add that a Krasnosel{\cprime}slki{\u\i}--Schaefer type result similar to Theorem~\ref{thm:krasnosielskii} can be found in~\cite{BBZ}, where the mapping $F_2$, defined on the whole Banach space $X$, is required to be completely continuous (that is, $F_2$ is continuous and maps bounded sets into relatively compact ones) rather than compact (see~\cite{BBZ}*{Theorem~4.2}). 
\end{remark}

\section{Existence results via a Leggett-Williams type theorem}
\label{sec:LW}

In this Section we are going to study the existence of continuous solutions to Hammerstein integral equations with kernels which may change sign. Let $\Omega$ be an open and bounded subset of $\mathbb R^n$ and let us consider the following Hammerstein integral equation
\begin{equation}\label{eq:int_H}
\lambda x(t)=\int_{\overline{\Omega}} k(t,s)f(s,x(s)) \textup ds,\quad t\in\overline{\Omega},
\end{equation}
where $k\colon\overline{\Omega}\times \overline{\Omega}\to \mathbb R$, $f\colon\overline{\Omega}\times \mathbb R\to \mathbb [0,+\infty)$ and $\lambda \neq 0$.

Before we proceed to the main part of this Section, let us recall a Leggett--Williams type theorem for Hammerstein integral equations with non-negative kernels.

\begin{theorem}[see~\cite{BK}*{Theorem~12}~and~\cite{LW}*{Theorem~2}]\label{thm:BK}
Let $\Omega$ be an open and bounded subset of $\mathbb R^n$, and let the functions $k\colon\overline{\Omega}\times \overline{\Omega}\to[0,+\infty)$ and $f\colon\overline{\Omega}\times [0,r]\to [0,+\infty)$ be continuous for some $r>0$. Moreover, suppose that there exist positive numbers $\delta_1,\delta_2,m$,  and a closed set $\Omega_0\subseteq\overline{\Omega}$ of positive Lebesgue measure such that\textup:
\begin{enumerate}[label=\textup{(\roman*)}, itemsep=5pt]
	\item\label{ass:BK1} $\displaystyle \int_{\Omega_0} k(t,s)\textup dt\geq \delta_1$  for each $s\in\Omega_0$\textup;
	\item\label{ass:BK2} $\displaystyle \int_{\Omega_0} k(t,s)\textup dt\geq \delta_2k(u,s)$  for each $(u,s)\in\overline{\Omega}\times \overline{\Omega}$\textup;
	\item $f(t,x)>0$ if $m\mu(\Omega_0)^{-1/p}\leq x\leq r$ and $t\in\Omega_0$, where $p \in [1,+\infty)$\textup;
	\item $0<m\leq r\delta_2\mu(\Omega_0)^{-1/q}$, where $q \in (1,+\infty]$ is such that $p^{-1}+q^{-1}=1$.
\end{enumerate}
Then there exists a positive parameter $\lambda>0$ such that the Hammerstein integral equation~\eqref{eq:int_H} admits a continuous and positive solution $x \colon \overline{\Omega} \to [0,r]$ such that
\begin{equation}\label{eq:LW_semi}
\biggl(\int_{\Omega_0}[x(t)]^p \textup dt\biggr)^{1/p}=m.
\end{equation}
\end{theorem}

\begin{remark}
Assume that $k \colon \overline{\Omega} \times \overline{\Omega} \to [0,+\infty)$ is continuous. Let us observe that if we define the mapping $\Phi \colon \overline{\Omega} \to [0,+\infty)$ by the formula
\begin{equation}\label{eq:phi}
 \Phi(s) = \delta^{-1}_2 \int_{\Omega_0} k(t,s) \textup dt, \qquad s \in \overline{\Omega},
\end{equation}
then the assumptions~\ref{ass:BK1} and~\ref{ass:BK2} of Theorem~\ref{thm:BK} may be restated as follows:
\begin{enumerate}[label={(\underline{\roman*})}, itemsep=5pt]
	\item\label{it:j} $\Phi(s) \geq \delta_1 \delta_2^{-1}$ for each $s\in\Omega_0$\textup;
	\item\label{it:jj} $k(u,s) \leq \Phi(s)$  for each $(u,s)\in\overline{\Omega}\times \overline{\Omega}$\textup.
\end{enumerate}

On the other hand, if there exists a mapping $\Phi \colon \overline{\Omega} \to [0,+\infty)$ which satisfies the above properties~\ref{it:j}-\ref{it:jj} (with the equality sign in~\eqref{eq:phi} replaced by the `less than or equal to' sign), then
\[
 \delta_2 k(u,s) \leq \delta_2  \Phi(s) \leq \int_{\Omega_0} k(t,s) \textup dt \quad \text{for $(u,s) \in \overline{\Omega}\times \overline{\Omega}$}
\]
and
\[
 \int_{\Omega_0} k(t,s) \textup dt \geq \delta_2 \Phi(s) \geq \delta_1 \quad \text{for $s \in \Omega_0$}.
\]
This shows that in this case the kernel $k$ satisfies the assumptions~\ref{ass:BK1} and~\ref{ass:BK2} of Theorem~\ref{thm:BK}.

We will use the above observation in the main result of this Section.
\end{remark}

Now, we will prove an extension of Theorem~\ref{thm:BK} for kernels that may change sign. Clearly, since the kernels are allowed to take negative values, one cannot expect, in general, to prove the existence of positive solutions. However, we will still be able to pinpoint the localization of the solutions by means of the supremum norm and the integral semi-norm (cf. formula~\eqref{eq:LW_semi}).

Let us assume that for some $r>0$ there exists a set $\Omega_0 \subseteq \overline{\Omega}$ with positive Lebesgue measure, together with constants $\vartheta  \in (0,1)$, $\eta_1, \eta_2,m,c \in (0,+\infty)$ and $p \in [1,+\infty)$ such that the functions $f \colon \overline{\Omega} \times [-r,r] \to [0,+\infty)$ and $k \colon \overline{\Omega} \times \overline{\Omega} \to \mathbb R$ satisfy the following conditions:
\begin{enumerate}[label=\textup{(A$_{\arabic*}$)}]
 \item\label{ass:A1} $f$ satisfies the Carath\'eodory conditions, that is,
   \begin{enumerate}[label=\textup{(\roman*)}]
	   \item for every $u \in [-r,r]$ the function $t \mapsto f(t,u)$ is Lebesgue measurable;
		 \item for a.e. $t \in \overline{\Omega}$ the function $u \mapsto f(t,u)$ is continuous;
		 \item\label{ass:A1iii} there exists a Lebesgue measurable function $g_r \colon \overline{\Omega} \to [0,+\infty)$ such that for a.e. $t \in \overline{\Omega}$ we have $f(t,u) \leq g_r(t)$ for all $u \in [-r,r]$;
	 \end{enumerate}

\item\label{ass:A1_iv} for a.e. $t \in \Omega_0$ we have $\inf \bset{f(t,u)}{\vartheta m\mu(\Omega_0)^{-1/p} \leq \abs{u} \leq r} \geq \eta_1$;	
	
\item\label{ass:A2} $k$ is Lebesgue measurable on $\overline{\Omega} \times \overline{\Omega}$, and for every $\tau \in \overline{\Omega}$ the function $s \mapsto k(\tau,s)$ is Lebesgue measurable\footnote{The fact that $s \mapsto k(\tau,s)$ is Lebesgue measurable for every $\tau \in \overline{\Omega}$, in general, does not follow from the Lebesgue measurability of $k$ on $\overline{\Omega} \times \overline{\Omega}$; the latter condition implies only that \emph{almost all} vertical sections are Lebesgue measurable (cf.~\cite{Lojasiewicz}*{Section~6.3}).} on $\overline{\Omega}$ and 
\[
\lim_{t \to \tau} \int_{\overline{\Omega}} \abs{k(t,s) - k(\tau,s)} g_r(s) \textup ds =0;
\] 

\item\label{ass:A3} there exists a Lebesgue measurable function $\Phi \colon \overline{\Omega} \to [0,+\infty)$ such that
\begin{enumerate}[label=\textup{(\roman*)}]
 \item\label{ass:A3i} $\Phi(t) \geq \eta_2$ for a.e. $t \in \Omega_0$;
 \item\label{ass:A3ii} for every $t \in \overline{\Omega}$ we have $\abs{k(t,s)} \leq \Phi(s)$  for a.e. $s \in \overline{\Omega}$;
 \item\label{ass:A3iii} $\displaystyle c \Phi(s) \leq \int_{\Omega_0} k(t,s) \textup dt$ for a.e. $s \in \overline{\Omega}$; 
\end{enumerate}

\item\label{ass:A4} $\displaystyle \int_{\overline{\Omega}} \Phi(s) g_r(s) \textup ds <+\infty$;

\item\label{ass:A5} $0<m\leq rc\mu(\Omega_0)^{-1/q}$, where $q \in (1,+\infty]$ is such that $p^{-1}+q^{-1}=1$.
\end{enumerate}

\begin{theorem}\label{th:LW3}
Suppose that the above assumptions hold. Then there exists a parameter $\lambda>0$ such that the Hammerstein integral equation~\eqref{eq:int_H} admits a continuous solution $x \colon \overline{\Omega} \to [-r,r]$ such that
\[
 \biggl(\int_{\Omega_0} \abs{x(t)}^p \textup dt\biggr)^{1/p} =m.
\] 
\end{theorem}



\begin{proof}
Let us set
\[
 C\zdef \bgset{x \in C(\overline{\Omega})}{ \int_{\Omega_0} x(t) \textup dt \geq c\norm{x}_{\infty}}
\]
and
\[
 \absp{x}=\biggl( \int_{\Omega_0} \abs{x(t)}^p \textup dt \biggr)^{1/p} \qquad \text{for $x \in C(\overline{\Omega})$}.
\]
Clearly, $C$ is a cone in $C(\overline{\Omega})$, and moreover, it is easy to check that the assumption~\ref{it:ar_2} of Theorem~\ref{thm:LW2} holds with $M=c^{-1}\mu(\Omega_0)^{1/q}$ (cf.~\cite{HSt}*{Theorem 13.17}).

Let $F$ be the mapping defined on $C(\theta,r)$ by
\begin{equation}\label{eq:F_LW}
 F(x)(t) = \int_{\overline{\Omega}} k(t,s) f(s,x(s)) \textup ds, \qquad t \in \overline{\Omega}.
\end{equation}
It can be shown that $F$ is a compact mapping from $C(\theta,r)$ into $C(\overline{\Omega})$ (cf.~\cite{M}*{Proposition~3.1, p.~164}). Furthermore, if $x \in C(\theta,r)$, then given any $u \in \overline{\Omega}$ we have
\begin{align*}
 c\abs{F(x)(u)} & \leq \int_{\overline{\Omega}} c\abs{k(u,s)}f(s,x(s)) \textup ds  \leq \int_{\overline{\Omega}}c\Phi(s)f(s,x(s)) \textup ds\\[2mm]
 & \leq \int_{\overline{\Omega}} \biggl( \int_{\Omega_0} k(t,s) \textup dt \biggr) f(s,x(s)) \textup ds\\[2mm]
 & = \int_{\Omega_0}\biggl( \int_{\overline{\Omega}} k(t,s) f(s,x(s)) \textup ds\biggr) \textup dt = \int_{\Omega_0} F(x)(t) \textup dt,
\end{align*}
which shows that $F(C(\theta,r)) \subseteq C$.

Observe that from the assumptions imposed on $k$ it follows that $c \leq \mu(\Omega_0)$, and hence if we define $x(t)=m\mu(\Omega_0)^{-1/p}$ for $t \in \overline{\Omega}$, we see that
\[
 \abs{x(t)} = m\mu(\Omega_0)^{-1/p} \leq cr\mu(\Omega_0)^{-1} \leq r \qquad \text{and} \qquad c\norm{x}_{\infty} = cm\mu(\Omega_0)^{-1/p} \leq \int_{\Omega_0} x(t)\textup dt.
\]
Thus we have $x \in C(\theta,r)$ and $\absp{x}=m$. 

Finally, we shall show that the assumption~\ref{it:LW2_4} of Theorem~\ref{thm:LW2} holds. Assume that $x \in C(\theta,r)$ is such that $\absp{x}=m$. Let us set
\[
 \Omega_1 = \bset{t \in \Omega_0}{\vartheta m \mu(\Omega_0)^{-1/p} \leq \abs{x(t)} \leq r}.
\]
Then we have
\begin{align*}
 m^p &= \int_{\Omega_0} \abs{x(t)}^p \textup dt = \int_{\Omega_1} \abs{x(t)}^p \textup dt +  \int_{\Omega_0\setminus \Omega_1} \abs{x(t)}^p \textup dt \leq r^p\mu(\Omega_1) + \dfrac{\vartheta^p m^p}{\mu(\Omega_0)} \cdot \mu(\Omega_0 \setminus \Omega_1)\\[2mm]
& \leq r^p \mu(\Omega_1) + \vartheta^p m^p.
\end{align*}
Thus we obtain $\mu(\Omega_1) \geq m^p(1-\vartheta^p)r^{-p}>0$. Hence we have
\begin{align*}
 \absp{F(x)} &= \biggl( \int_{\Omega_0} \abs{F(x)(t)}^p \textup dt\biggr)^{1/p} \geq \mu(\Omega_0)^{-1/q}  \int_{\Omega_0} \abs{F(x)(t)}  \textup dt \geq \mu(\Omega_0)^{-1/q}  \int_{\Omega_0} F(x)(t)  \textup dt\\[2mm]
& = \mu(\Omega_0)^{-1/q} \int_{\overline{\Omega}} \biggl( \int_{\Omega_0} k(t,s) \textup dt \biggr) f(s,x(s)) \textup ds \geq c \mu(\Omega_0)^{-1/q} \int_{\Omega_0}\Phi(s)f(s,x(s)) \textup ds\\[2mm]
& \geq c\eta_1 \mu(\Omega_0)^{-1/q} \int_{\Omega_1}\Phi(s)\textup ds \geq c \eta_1 \eta_2 \mu(\Omega_0)^{-1/q}m^p(1-\vartheta^p)r^{-p}\defz \delta>0.
\end{align*}
To end the proof it suffices to apply Theorem~\ref{thm:LW2}.
\end{proof}

\begin{remark}\label{rem:LW2}
Let us note that in the simplest, yet very common, case when the functions $f \colon \overline{\Omega} \times [-r,r] \to [0,+\infty)$ and $k \colon \overline{\Omega} \times \overline{\Omega} \to \mathbb R$ are continuous, the assumptions~\ref{ass:A1}--\ref{ass:A5} are implied by the following set of conditions:

`There exists a closed set $\Omega_0 \subseteq \overline{\Omega}$ with positive Lebesgue measure, together with constants $m \in (0,+\infty)$ and $p \in [1,+\infty)$ such that:
\begin{enumerate}[label=\textup{(B$_{\arabic*}$)}]
 \item $f(t,u)>0$ for $t \in \Omega_0$ and $m\mu(\Omega_0)^{-1/p}\leq \abs{u}\leq r$;\\
 
 \item $\displaystyle \min_{s \in \overline{\Omega}} \int_{\Omega_0} k(t,s)\textup dt >0$;

 \item $\displaystyle 0<m \leq r \mu(\Omega_0)^{-1/q} \norm{k}_{\infty}^{-1} \min_{s \in \overline{\Omega}} \int_{\Omega_0} k(t,s)\textup dt$.'
\end{enumerate}
Indeed, \ref{ass:A1}--\ref{ass:A2} and \ref{ass:A5} trivially hold; \ref{ass:A1}~\ref{ass:A1iii} and~\ref{ass:A2} with $g_r(t)=\norm{f}_{\infty}$ for $t \in \overline{\Omega}$, whereas $\ref{ass:A5}$ with $c:=\norm{k}_{\infty}^{-1} \min_{s \in \overline{\Omega}} \int_{\Omega_0} k(t,s)\textup dt$. To see that the assumptions~\ref{ass:A3}--\ref{ass:A4} also hold it suffices to set
\[
 \Phi(s)=\frac{\norm{k}_{\infty}  \int_{\Omega_0} k(t,s)\textup dt}{\min_{s \in \overline{\Omega}} \int_{\Omega_0} k(t,s)\textup dt}\quad \text{for $s \in \overline{\Omega}$}.
\]
\end{remark}

\subsection{Periodic BVPs}\label{sec:BVP_LW} 
In this short Section, we would like to show how to apply Leggett--Williams type theorems to proving the existence of solutions to BVPs. Therefore, let us consider the following periodic BVP
\begin{gather}
x''(t)+\omega^2 x(t)=\lambda f(t,x(t)), \quad t \in [0,1], \label{eq:BVP1}\\
x(0)=x(1), \quad x'(0)=x'(1),\label{eq:BVP2}
\end{gather}
where $\lambda \neq 0$ and $\omega$ is a positive constant such that $\omega \neq 2n\pi$ for $n \geq 1$. For simplicity, we assume that $f \colon [0,1] \times \mathbb R \to \mathbb [0,+\infty)$ is continuous. It can be checked that each continuous solution of the following nonlinear Hammerstein integral equation
\begin{equation}\label{eq:BVP_H}
 x(t) = \lambda \int_0^1 k(t,s)f(s,x(s))\textup ds, \quad t \in [0,1],
\end{equation}
where
\begin{equation}\label{eq:green_c}
 k(t,s) = \begin{cases}
           \dfrac{\cos[\omega(\tfrac{1}{2}-t+s)]}{2\omega \sin(\tfrac{1}{2}\omega)}, \quad \text{if $0\leq s \leq t \leq 1$},\\[5mm]
					 \dfrac{\cos[\omega(\tfrac{1}{2}-s+t)]}{2\omega \sin(\tfrac{1}{2}\omega)}, \quad \text{if $0\leq t < s \leq 1$},
					\end{cases}
\end{equation}
is a solution\footnote{Let us recall that by a solution to the BVP~\eqref{eq:BVP1}--\eqref{eq:BVP2} we understand a real-valued twice continuously differentiable function defined on $[0,1]$ which satisfies both the equation~\eqref{eq:BVP1} and the periodic BCs~\eqref{eq:BVP2}.} to the BVP~\eqref{eq:BVP1}--\eqref{eq:BVP2}. 

Furthermore, let us note that if, for example, $\omega=\frac{3}{2}\pi$, then $k(t,t)=-\frac{1}{3\pi}<0$ for $t \in [0,1]$. This means that to study the existence of solutions of the BVP~\eqref{eq:BVP1}--\eqref{eq:BVP2} with $\omega=\frac{3}{2}\pi$ we cannot apply results for mappings in cones which require $k$ to be non-negative on a rectangle of the form $[a,b]\times [0,1]$ such as, for example, Theorem~\ref{thm:BK} or the results in~\cites{W2}.

\begin{theorem}\label{thm:existence_LW}
Let $\omega$ be a positive constant such that $\omega \neq 2n\pi$ for $n\geq 1$. Moreover, let $r \geq 1$ and let the continuous function $f\colon [0,1]\times [-r,r] \to \mathbb [0,+\infty)$ be such that $f(t,u)>0$ if $0 \leq t \leq 1$ and $2\omega^{-1}\abs{\sin(\frac{1}{2}\omega)} \leq \abs{u} \leq r$. Then for every $p \in [1,+\infty)$ there exists a positive parameter $\lambda:=\lambda(\omega, p)$ such that the BVP~\eqref{eq:BVP1}--\eqref{eq:BVP2} has a solution $x\colon [0,1] \to [-r,r]$ with $\bigl(\int_0^1 \abs{x(t)}^p \textup dt\bigr)^{1/p}=2\omega^{-1}\abs{\sin(\frac{1}{2}\omega)}$.
\end{theorem}

\begin{proof}
Since $\int_0^1 k(t,s)\textup dt = \omega^{-2}$ for every $s \in [0,1]$, the proof of Theorem~\ref{thm:existence_LW} is a direct consequence of Theorem~\ref{th:LW3} if one sets: $\overline{\Omega}=\Omega_0=[0,1]$ and $m=2\omega^{-1}\abs{\sin(\frac{1}{2}\omega)}$ (cf. Remark~\ref{rem:LW2}).
\end{proof}

\begin{remark}\label{rem:inc}
The existence of solutions to the BVP~\eqref{eq:BVP1}--\eqref{eq:BVP2} was also studied in, for example,~\cites{W2, G} (more general periodic BVPs were investigated in~\cite{Ma}), where under some growth conditions on the function $f$, it was shown that the BVP~\eqref{eq:BVP1}--\eqref{eq:BVP2} admits a positive solution, provided $\lambda=1$ and $\omega \in (0,\pi]$.

Let us note that for such values of $\omega$ also our result ensures the existence of positive solutions to the BVP~\eqref{eq:BVP1}--\eqref{eq:BVP2} (this follows from the fact that $k(t,s) \geq 0$ for $(t,s) \in [0,1] \times [0,1]$, if $\omega \in (0,\pi]$).

However, our result and the results obtained in~\cites{W2, G} are in a sense `incomparable' (even in the case $\omega \in (0,\pi]$), since on the one hand our assumptions imposed on the function $f$ are less restrictive than those introduced in the above-mentioned articles  (for example, we do not require $f$ to satisfy certain growth conditions at zero and at infinity), but on the other hand we are able to prove the existence of solutions to the BVP~\eqref{eq:BVP1}--\eqref{eq:BVP2} for some positive parameter $\lambda$, whereas in~\cites{W2, G} $\lambda=1$.

What is worth mentioning is that the fact that Theorem~\ref{thm:existence_LW} ensures the existence of solutions to the BVP~\eqref{eq:BVP1}--\eqref{eq:BVP2} only for some $\lambda>0$ is something that cannot be avoided and is not a consequence of the approach, but is, one might say, `forced' by the additional constraints imposed on the sought solution and the problem itself. To better understand this phenomenon let us consider the BVP~\eqref{eq:BVP1}--\eqref{eq:BVP2} with $f(t,u)=u^2$, $\omega=\pi$ and $\lambda=1$. It is easy to see that the integral operator $F$ corresponding to the considered BVP in the integral form~\eqref{eq:BVP_H} (cf. formula~\eqref{eq:F_LW}) is a contraction with respect to the supremum norm (with Lipschitz constant $r/\pi$) on the closed ball $B_{C[0,1]}(\theta,r)$, where $r \in (0,\pi)$, and therefore it has at most one fixed point in this ball. However, we know that $F(\theta)=\theta$, which means that there are no solutions to the BVP in question with supremum norm not exceeding $r$ other than the zero solution.
\end{remark}

\section{Existence results via Krasnosel{\cprime}slki{\u\i}'s theorem}
\label{sec:Kras}

The aim of this Section is to prove the existence of $CBV$-solutions to the following perturbed nonlinear Hammerstein integral equation
\begin{equation}\label{eq:bvp_hammerstein}
 x(t)=\alpha[x]v(t) + \beta[x]w(t) + \lambda \int_0^1 k(t,s)f(s,x(s))\textup ds, \quad t \in [0,1],
\end{equation}
where $\lambda \in \mathbb R$, using Krasnosel{\cprime}slki{\u\i}'s fixed-point theorem. As in Section~\ref{sec:LW}, we are also going to provide applications of our result to some BVPs.

Before we proceed further, let us make the following assumptions on the functionals $\alpha, \beta \in CBV^*[0,1]$ and the functions $v,w \in CBV[0,1]$:
\begin{enumerate}[label=\textup{(A$_{\arabic*}$)}]
\setcounter{enumi}{6}
 \item\label{a1} $\alpha[e]=\beta[e]=0$; here $e$ denotes the constant function given by $e(t)=1$ for $t\in [0,1]$; 
 
 \item\label{a3} $\babs{\alpha[v]-\beta[v]}<1$;

 \item\label{a2} $v(t)+w(t)=1$ for every $t \in [0,1]$.
\end{enumerate}

\noindent Furthermore, let us assume that the nonlinearity $f \colon [0,1] \times \mathbb R \to \mathbb R$ and the kernel $k \colon [0,1] \times [0,1] \to \mathbb R$ satisfy the following conditions:
\begin{enumerate}[label=\textup{(A$_{\arabic*}$)}, resume]
\item\label{a4} $f$ satisfies the Carath\'eodory conditions, that is,
   \begin{enumerate}[label=\textup{(\roman*)}]
 \item for every $u \in \mathbb R$ the function $t \mapsto f(t,u)$ is Lebesgue measurable;
 \item for a.e. $t \in [0,1]$ the function $u \mapsto f(t,u)$ is continuous;
 \item\label{a4c} there exist a non-decreasing function  $\psi \colon [0,+\infty) \to [0,+\infty)$ and an $L^p$-function $\phi \colon [0,1] \to [0,+\infty)$  with $p \in (1,+\infty]$ such that $\abs{f(t,u)} \leq \phi(t)\psi(\abs{u})$ for $t \in [0,1]$ and $u \in \mathbb R$;
\end{enumerate}

\item\label{a4d}  $\lim_{r \to +\infty} \psi(r)/r=0$;

\item\label{a5} for every $t \in [0,1]$ the function $s \mapsto k(t,s)$ is an $L^q$-function; here $q \in [1,+\infty)$ is such that $p^{-1}+q^{-1}=1$;
 
\item\label{a5a} there exists an $L^q$-function $m \colon [0,1] \to [0,+\infty)$ such that $\var_{[0,1]} k(\cdot,s) \leq m(s)$ for a.e. $s \in [0,1]$;

\item\label{a5d}  for every $\tau \in [0,1]$ we have
\[
 \lim_{t \to \tau} \int_0^1 \abs{k(t,s)-k(\tau,s)}\phi(s) \textup ds=0.
\]
\end{enumerate} 

\begin{remark}
If the kernel $k$ satisfies~\ref{a5a}, then in~\ref{a5} we do not need to assume that the function $s \mapsto k(t,s)$ is an $L^q$-function for every $t \in [0,1]$; it is enough to assume that all the vertical sections are Lebesgue measurable and only one of them is integrable with $q$th power (cf.~\cite{BGK}*{Remark~7} or~\cite{Bug}*{Remark~7.3}).
\end{remark}

\begin{remark}\label{rem:k}
Let us note that the assumptions~\ref{a5}--\ref{a5d} hold if, for example, the kernel $k \colon [0,1] \times [0,1] \to \mathbb R$ is continuous and satisfies the following condition:
\begin{enumerate}[label=\textup{(B$_{\arabic*}$)}]
\setcounter{enumi}{3}
 \item\label{it:B_kern} there exists an $L^q$-function $m \colon [0,1] \to [0,+\infty)$ such that $\abs{k(t,s)-k(\tau,s)} \leq m(s)\abs{t-\tau}$ for all $s,t,\tau \in [0,1]$.
\end{enumerate}
\end{remark}

We set
\begin{equation}\label{eq:f_1}
 F_1(x)(t)=\alpha[x]v(t) + \beta[x]w(t), \quad t \in [0,1]
\end{equation}
and
\begin{equation}\label{eq:f_2}
F_2(x)(t)=\int_0^1 k(t,s)f(s,x(s)) \textup ds, \quad t \in [0,1],
\end{equation}
so that~\eqref{eq:bvp_hammerstein} with $\lambda=1$ takes the following operator form
\[
 x = F_1(x) + F_2(x).
\] 
In order to show that the set of fixed points of $F_1+F_2$ is non-empty (which would obviously imply that the Hammerstein integral equation~\eqref{eq:bvp_hammerstein} with $\lambda=1$ has a solution) we are going to apply Theorem~\ref{thm:krasnosielskii}, and therefore we begin with proving that the operators $F_1$ and $F_2$ have the required properties.

\begin{lemma}\label{lem:F_1_properties}
Suppose that the assumptions~\ref{a1} and~\ref{a2} hold. Then $F_1 \colon CBV[0,1] \to CBV[0,1]$ given by~\eqref{eq:f_1} is a bounded linear operator with 
\[
\norm{F_1}\leq \norm{\alpha}+\norm{\alpha-\beta}\cdot\norm{w}_{BV}
\]
and
\begin{equation}\label{eq:estimate_F1n}
 \norm{F_1^{n+2}}\leq \bnorm{\alpha-\beta}\cdot \babs{\alpha[v]-\beta[v]}^{n} \cdot \bnorm{\alpha[v]v+\beta[v]w}_{BV} \quad \text{for $n \geq 0$}.\footnote{If $n=0$, then, by definition, we take $\babs{\alpha[v]-\beta[v]}^{n}=1$, even if $\babs{\alpha[v]-\beta[v]}=0$.}
\end{equation}
In particular, $r(F_1)\leq \babs{\alpha[v]-\beta[v]}$.
\end{lemma}

\begin{proof}
The proof of Lemma~\ref{lem:F_1_properties} is straightforward (although tedious) and hence it will be omitted.
\end{proof}

\begin{remark}\label{rem:properties_F1}
Let us note that in order to prove the estimate for $\norm{F_1}$ one does not have to use the assumption~\ref{a1}. This assumption is used in showing~\eqref{eq:estimate_F1n}.
\end{remark}

\begin{lemma}\label{lem:F_2_properties}
If the assumptions~\ref{a4} and~\ref{a5}--\ref{a5d} hold, then the mapping $F_2 \colon BV[0,1] \to CBV[0,1]$ given by~\eqref{eq:f_2} is completely continuous.
\end{lemma}

\begin{proof}
First, we shall show that for every $x \in BV[0,1]$ the function $F_2(x)$ is well-defined. So let us fix $x \in BV[0,1]$. Then for every $s \in [0,1]$ we have $\abs{f(s,x(s))} \leq \phi(s) \psi(\norm{x}_{BV})$. Therefore, the function $s \mapsto k(t,s)f(s,x(s))$ is Lebesgue integrable for every $t \in [0,1]$; Lebesgue measurability of the above function follows from the Carath\'eodory conditions imposed on $f$ and the fact that $x$ is Lebesgue measurable (cf.~\cite{ABM}*{Theorem~1.5}).

Now, we are going to prove that $F_2$ maps the space $BV[0,1]$ into $CBV[0,1]$. Fix $x \in BV[0,1]$. If $0=t_0 < \ldots < t_n=1$ is an arbitrary finite partition of the interval $[0,1]$, then, in view of the assumption~\ref{a5a}, we infer that
\begin{equation}\label{eq:k_m}
 \sum_{i=1}^n \abs{k(t_i,s) - k(t_{i-1},s)} \leq m(s) \qquad \text{for a.e. $s\in [0,1]$},
\end{equation}
and so
\begin{align*}
 \sum_{i=1}^n \abs{F_2(x)(t_i)-F_2(x)(t_{i-1})} & \leq \int_0^1 \sum_{i=1}^n \abs{k(t_i,s)-k(t_{i-1},s)}\phi(s)\psi(\norm{x}_{BV})\textup ds \\
& \leq \psi(\norm{x}_{BV})\int_0^1 m(s)\phi(s)\textup ds.
\end{align*}
Thus $\var_{[0,1]} F_2(x) \leq \psi(\norm{x}_{BV})\int_0^1 m(s)\phi(s)\textup ds$. This shows that $F_2(x) \in BV[0,1]$. Continuity of the function $F_2(x)$ is a simple consequence of the assumption~\ref{a5d}. Hence, $F_2(x) \in CBV[0,1]$, which proves our claim.

Finally, we will show that $F_2$ is completely continuous. Suppose that $(x_n)_{n \in \mathbb N}$ is an arbitrary sequence of $BV$-functions which is convergent to some $x \in BV[0,1]$ with respect to the $BV$-norm\footnote{Let us recall that if a sequence of $BV$-functions is convergent to a function $x\in BV[0,1]$ with respect to the $BV$-norm, then it is uniformly convergent to $x$ on the whole interval $[0,1]$.}. In particular, $(x_n)_{n \in \mathbb N}$ is bounded, which means that $\sup_{n \in \mathbb N} \norm{x_n}_{BV} \leq R_1$ and $\norm{x}_{BV}\leq R_1$ for some $R_1>0$. Using the fact that the estimate~\eqref{eq:k_m} is true for an arbitrary finite partition of the interval $[0,1]$, we see that
\[
 \norm{F_2(x_n)-F_2(x)}_{BV} \leq \int_0^1 \bigl(m(s)+\abs{k(0,s)}\bigr)\abs{f(s,x_n(s))-f(s,x(s))}\textup ds \qquad \text{for $n \in \mathbb N$}.
\]
Since all the integrands on the right-hand side of the above formula can be majorized by the Lebesgue integrable function $s \mapsto 2\bigl(m(s)+\abs{k(0,s)}\bigr)\phi(s)\psi(R_1)$, which, clearly, does not depend on $n$, by the Lebesgue dominated convergence theorem, we infer that $\norm{F_2(x_n)-F_2(x)}_{BV}\to 0$ as $n \to +\infty$. This shows that $F_2$ is continuous. 

To prove that $F_2$ maps bounded subsets of $BV[0,1]$ into relatively compact subsets of $CBV[0,1]$, we will use similar techniques to those used in~\cite{BGK}. Fix $R_2>0$ and let $(x_n)_{n \in \mathbb N}$ be an arbitrary sequence of $BV$-functions such that $\norm{x_n}_{BV} \leq R_2$ for $n \in \mathbb N$. By Helly's selection theorem there exists a subsequence $(x_{n_k})_{k \in \mathbb N}$ of $(x_n)_{n \in \mathbb N}$ and a function $x \in BV[0,1]$ such that $x_{n_k} \to x$ pointwise on $[0,1]$ and $\norm{x}_{BV}\leq R_2$ (see~\cite{ABM}*{Theorem~1.11} or~\cite{Lojasiewicz}*{1.4.5}). From the first part of the proof it follows that $y:=F_2(x) \in CBV[0,1]$. Reasoning similar to that used in the proof of the continuity of $F_2$ shows that $F_2(x_{n_k}) \to y$ in $CBV[0,1]$ as $k\to +\infty$, which means that the sequence $(F_2(x_{n}))_{n \in \mathbb N}$ has a convergent subsequence. Therefore, $F_2$ is completely continuous.
\end{proof}

Now, we are ready to prove the first main result of this Section.

\begin{theorem}\label{thm:solution}
Under the assumptions~\ref{a1}--\ref{a5d} the perturbed nonlinear Hammerstein integral equation~\eqref{eq:bvp_hammerstein} with $\lambda=1$ has a $CBV$-solution.
\end{theorem}

\begin{proof}
First, let us observe that by the assumption~\ref{a4d} it is possible to find a positive number $r$ such that
\begin{equation}\label{eq:psi}
 \psi(r)\leq r c^{-1}\biggl(1+\int_0^1 \bigl(m(s)+\abs{k(0,s)}\bigr) \phi(s)\textup ds \biggr)^{-1},
\end{equation}
where
\begin{equation}\label{eq:noma_i_b}
 c:=1+\norm{\alpha}+\norm{\alpha-\beta}\cdot \norm{w}_{BV}+  \frac{\norm{\alpha-\beta}\cdot \bnorm{\alpha[v]v+\beta[v]w}_{BV}}{1-\babs{\alpha[v]-\beta[v]}}.
\end{equation}
Set $M:=B_{CBV}(\theta,r)$ and consider the mappings $F_1 \colon CBV[0,1] \to CBV[0,1]$ and $F_2 \colon M \to CBV[0,1]$ given by the formulas~\eqref{eq:f_1} and~\eqref{eq:f_2}, respectively. From Lemmas~\ref{lem:F_1_properties} and~\ref{lem:F_2_properties}
it follows that $F_1$ and $F_2$ satisfy the assumptions~\ref{it:krasnosielski_1} and~\ref{it:krasnosielski_2} of Theorem~\ref{thm:krasnosielskii}. Moreover, for every $x \in M$, we have
\[
 \norm{F_2(x)}_{BV} \leq \psi(r)\cdot  \int_0^1 \bigl(m(s)+\abs{k(0,s)}\bigr) \phi(s)\textup ds
\]
(cf. the proof of Lemma~\ref{lem:F_2_properties}). This, together with the fact that
\[
 \norm{(I-F_1)^{-1}} = \bgnorm{\sum_{n=0}^{\infty} F_1^n} \leq \sum_{n=0}^{\infty} \norm{F_1^n} = c
\]
(cf. Lemma~\ref{lem:F_1_properties}), yields
\[
 \norm{F_2(x)}_{BV}\leq \frac{r}{\norm{(I-F_1)^{-1}}} \qquad \text{for $x \in M$}.
\]
To end the proof it suffices to apply Remark~\ref{rem:krasosielski} and Theorem~\ref{thm:krasnosielskii}.
\end{proof}

\begin{remark}\label{rem:assumtpion_weak}
From the proof of Theorem~\ref{thm:solution} it follows that the condition~\ref{a4d} can be replaced by a weaker (but at the same time more technical) condition:
\begin{enumerate}[label=\textup{(B$_{\arabic*}$)}]
\setcounter{enumi}{4}
 \item\label{a11_} there exists $r>0$ such that $\psi(r)\leq r c^{-1}\bigl(1+\int_0^1 (m(s)+\abs{k(0,s)}) \phi(s)\textup ds \bigr)^{-1}$, where $c$ is given by~\eqref{eq:noma_i_b}.
\end{enumerate}
\end{remark}

\begin{corollary}\label{cor:1}
Suppose that the assumptions~\ref{a1}--\ref{a4} and~\ref{a5}--\ref{a5d} hold. Then there exists $\lambda_0 >0$ such that the perturbed nonlinear Hammerstein integral equation~\eqref{eq:bvp_hammerstein} has a $CBV$-solution for every $\lambda \in \mathbb R$ such that $\abs{\lambda}\leq \lambda_0$.
\end{corollary}

\begin{proof}
Let
\[
 \lambda_0 := \frac{1}{c(\psi(1)+1)}\biggl(1+\int_0^1 \bigl(m(s)+\abs{k(0,s)}\bigr)\phi(s)\textup ds \biggr)^{-1},
\]
where the number $c$ is given by~\eqref{eq:noma_i_b}. The claim follows from Remark~\ref{rem:assumtpion_weak} and  Theorem~\ref{thm:solution} if we replace $f$ and $\psi$ with $\lambda f$ and $\abs{\lambda} \psi$.
\end{proof}

It turns out that if we strengthen the condition~\ref{a3}, then~\ref{a1} is not required. Indeed, by a similar argument to the proof of Theorem~\ref{thm:solution}, in view of the classical Krasnosel{\cprime}slki{\u\i}'s theorem (see, for example,~\cites{Park, K}) and Lemma~\ref{lem:F_1_properties} (cf. also Remark~\ref{rem:properties_F1}), one can establish the following result.

\begin{theorem}\label{thm:solution2}
Suppose that the assumptions~\ref{a2}--\ref{a5d}  hold. If, additionally, the functionals $\alpha, \beta \in CBV^{*}[0,1]$ satisfy the following condition\textup:
\begin{enumerate}[label=\textup{(B$_{\arabic*}$)}]
\setcounter{enumi}{5}
 \item\label{b6} $\norm{\alpha} + \norm{\alpha-\beta}\cdot \norm{w}_{BV}<1$,
\end{enumerate}
then the perturbed nonlinear Hammerstein integral equation~\eqref{eq:bvp_hammerstein} with $\lambda=1$ has a $CBV$-solution.
\end{theorem}

\begin{corollary}\label{cor:hammerstein_ver_2}
Suppose that the conditions~\ref{a2}--\ref{a4},~\ref{a5}--\ref{a5d} and~\ref{b6} hold. Then there exists $\lambda_0 >0$ such that the perturbed nonlinear Hammerstein integral equation~\eqref{eq:bvp_hammerstein} has a $CBV$-solution for every $\lambda \in \mathbb R$ such that $\abs{\lambda}\leq \lambda_0$. 
\end{corollary}

\subsection{BVPs with non-local BCs} 
\label{ssec:BVP_nl}

In this Section, as an application of the abstract results for the perturbed nonlinear Hammerstein integral equation, we are going to study the existence of solutions to the following second-order differential equation
\begin{equation}\label{eq:bvp}
x''(t)=-\lambda f(t,x(t)), \qquad t \in [0,1],
\end{equation}
with the non-local BCs of the form
\begin{equation}\label{eq:bvp_bc}
 x(0)=\int_0^1 A(s)\textup dx(s) \qquad \text{and} \qquad x(1)=\int_0^1 B(s)\textup dx(s).
\end{equation}
For simplicity, as in Section~\ref{sec:BVP_LW}, we assume that the function $f \colon [0,1] \times \mathbb R \to \mathbb R$ is continuous. 

Before we proceed to the main theorems of this Section, we will discuss a class of functions for which the BCs~\eqref{eq:bvp_bc} are well-posed and we will prove a result concerning continuous linear functionals on $CBV[0,1]$.

Let us introduce the following notation. Given a number $\varepsilon>0$, we say that a bounded function $A \colon [a,b] \to \mathbb R$, where $-\infty<a<b<+\infty$, belongs to the family $\Omega_{\varepsilon}[a,b]$, if there exists $\delta>0$ such that $\osc_{[t,s]}A \leq \varepsilon$, whenever $t,s \in [a,b]$ are such that $0\leq s-t\leq \delta$; here the symbol $\osc_{[t,s]} A$ denotes the \emph{oscillation} of the function $A$ over the interval $[t,s]$, that is, $\osc_{[t,s]} A:=\sup_{t\leq \tau \leq \sigma \leq s}\abs{A(\sigma)-A(\tau)}$. Let us denote by $\Omega[0,1]$ the set of all bounded functions $A \colon [0,1] \to \mathbb R$ such that for every $\varepsilon>0$ there is $a \in (0,1)$ such that $A|_{[0,a]}\in \Omega_{\varepsilon}[0,a]$ and $A|_{[a,1]} \in BV[a,1]$. For simplicity, let us also set $\widehat{\Omega}[0,1]:=\Omega[0,1] \cup C[0,1] \cup BV[0,1]$.

\begin{example}
An example of a bounded function $A\colon [0,1] \to \mathbb R$ which belongs to $\widehat{\Omega}[0,1]$ but is neither continuous nor of bounded variation on the interval $[0,1]$ is given by the following formula
\[
 A(t)=\begin{cases}
       \frac{1}{n}, & \text{if $t \in (\frac{1}{n+1},\frac{1}{n})$, $n \in \mathbb N$},\\
			  0, & \text{otherwise}.
			\end{cases}
\]
\end{example}

\begin{lemma}\label{lem:RS_integral}
If $x \in CBV[0,1]$ and $A \in \widehat{\Omega}[0,1]$, then the Riemann--Stieltjes integral $\int_0^1 A(s)\textup dx(s)$ exists.
\end{lemma}

\begin{proof}
In view of~\cite{Lojasiewicz}*{Theorems~1.5.5 and~1.5.6} it is clear that we may assume that $A \in \Omega[0,1]$. We may also assume that $\norm{A}_{\infty}>0$ and $\var_{[0,1]}x>0$. Then,  given $\varepsilon>0$ there exists $a \in (0,1)$ such that $A|_{[0,a]} \in \Omega_{\eta}[0,a]$ and $A|_{[a,1]} \in BV[a,1]$, where $\eta:=\varepsilon(3\var_{[0,1]}x)^{-1}$. Let $\delta=\min\{\delta_1,\delta_2\}$, where $\delta_1 \in (0,1)$ is chosen according to the definition of the family $\Omega_{\eta}[0,a]$ and $\delta_2 \in (0,1)$ is such a number that 
\[
 \var_{[t,s]}x\leq \min\biggl\{\frac{\varepsilon}{6\norm{A}_{\infty}},\frac{\varepsilon}{3(1+\var_{[a,1]}A)}\biggr\},
\]
whenever $t,s \in [0,1]$ are such that $0 \leq s -t\leq \delta_2$; the number $\delta_2$ exists by the continuity of the function $t \mapsto \var_{[0,t]}x$ (cf.~\cite{Lojasiewicz}*{Theorem~1.3.4}).

If $0=t_0<t_1<\ldots<t_n=1$ is an arbitrary finite partition of the interval $[0,1]$ such that $\max_{1\leq i\leq n}\abs{t_i-t_{i-1}}\leq \delta$, then\footnote{If the upper summation limit is smaller than the lower one, then, by definition, the sum is equal to zero.}
\begin{align*}
 \sum_{i=1}^n \osc_{[t_{i-1},t_i]}A \cdot \var_{[t_{i-1}, t_i]}x &=  \sum_{i=1}^j \osc_{[t_{i-1},t_i]}A \cdot \var_{[t_{i-1}, t_i]}x +  \osc_{[t_{j},t_{j+1}]}A \cdot \var_{[t_{j}, t_{j+1}]} x +  \sum_{i=j+2}^n \osc_{[t_{i-1},t_i]}A \cdot \var_{[t_{i-1}, t_i]}x,
\end{align*}
where the index $j \in \{0,\ldots,n-1\}$ is chosen in such a way that $t_j\leq a<t_{j+1}$. Thus
\[
 \sum_{i=1}^n \osc_{[t_{i-1},t_i]}A \cdot \var_{[t_{i-1}, t_i]}x \leq \frac{\varepsilon \var_{[0,a]}x}{3\var_{[0,1]}x}  + \frac{2\varepsilon \norm{A}_{\infty}}{6\norm{A}_{\infty}} + \frac{\varepsilon \var_{[a,1]}A}{3(1+\var_{[a,1]}A)} \leq \varepsilon,
\] 
since $\osc_{[t_{i-1},t_i]}A \leq \var_{[t_{i-1},t_i]} A$ (see~\cite{Lojasiewicz}*{Formula~(1.3.4)}). This shows that the condition (B) of~\cite{Lojasiewicz}*{Theorem~1.5.2} is satisfied, and therefore the Riemann--Stieltjes integral $\int_0^1 A(s)\textup dx(s)$ exists.
\end{proof}

From Lemma~\ref{lem:RS_integral} and the properties of the Riemann--Stieltjes integral we have the following straightforward consequence.

\begin{corollary}\label{cor:RS_integral}
If $A \in \widehat{\Omega}[0,1]$, then the formula $x \mapsto \int_0^1 A(s)\textup dx(s)$ defines a continuous linear functional on $CBV[0,1]$.
\end{corollary}

\begin{remark}
Clearly, Lemma~\ref{lem:RS_integral} and Corollary~\ref{cor:RS_integral} would also be true, if in the definition of the class $\Omega[0,1]$ we assumed that $A|_{[a,1]} \in \Omega_{\varepsilon}[a,1]$ and $A|_{[0,a]}\in BV[0,a]$.
\end{remark}

Now, let us return to the BVP \eqref{eq:bvp}--\eqref{eq:bvp_bc}. It can be shown that if $A,B \in \widehat{\Omega}[0,1]$ and the function $f \colon [0,1]\times\mathbb R\to\mathbb R$ is continuous, then the BVP~\eqref{eq:bvp}--\eqref{eq:bvp_bc} is equivalent to the following perturbed nonlinear Hammerstein integral equation

\begin{equation}\label{eq:perturbed_Hammerstein_2}
x(t)= \int_0^1 (1-t)A(s)\textup dx(s) + \int_0^1 tB(s)\textup dx(s)+\lambda \int_0^1 k(t,s)f(s,x(s))\textup ds,\ t\in [0,1],
\end{equation}
where the kernel $k$ has the following form
\begin{equation}\label{eq:kernel_2}
k(t,s)=\begin{cases}
				s(1-t), & \text{if $0\leq s\leq t \leq 1$},\\
				t(1-s), & \text{if $0 \leq t <s \leq 1$},
			 \end{cases}
\end{equation}
that is, each twice continuously differentiable function $x\colon [0,1] \to \mathbb R$ which satisfies~\eqref{eq:bvp}--\eqref{eq:bvp_bc} is a $CBV$-solution to~\eqref{eq:perturbed_Hammerstein_2}, and vice-versa.

\begin{remark}\label{rem:kernel_BVP2}
It is easy to see that the kernel $k$ given by~\eqref{eq:kernel_2} is continuous and satisfies the condition~\ref{it:B_kern} with the function $m \colon [0,1] \to [0,+\infty)$ given by $m(s)=1$.
\end{remark}

\begin{theorem}\label{thm:bvp_nlc}
Let $f \colon [0,1]\times\mathbb R\to\mathbb R$ be a continuous function. Moreover, let the functions $A,B \in \widehat{\Omega}[0,1]$ be such that 
\[
 \bgabs{\int_0^1[A(s)-B(s)]\textup ds}<1.
\]
Then there exists a number $\lambda_0>0$ such that for any $\lambda \in \mathbb R$ satisfying $\abs{\lambda}\leq \lambda_0$ the BVP \eqref{eq:bvp}--\eqref{eq:bvp_bc} has a solution.
\end{theorem}

\begin{proof}
Since $f$ is continuous (and thus locally bounded), the function $\psi \colon [0,+\infty) \to [0,+\infty)$ given by
\[
 \psi(r):=\sup\bset{\abs{f(t,u)}}{\text{$t\in[0,1]$ and $u\in[-r,r]$}} \qquad \text{for $r\geq 0$}
\]
is well-defined and non-decreasing. Moreover, $k$ given by~\eqref{eq:kernel_2} is continuous and satisfies the condition~\ref{it:B_kern} with $m(s)=1$ (see Remark~\ref{rem:kernel_BVP2}). Therefore, the proof of Theorem~\ref{thm:bvp_nlc} is a direct consequence of Corollary \ref{cor:1}, Corollary~\ref{cor:RS_integral} and Remark~\ref{rem:k} if one sets $\phi(t)=1$, $v(t)=1-t$ and $w(t)=t$ for $t \in [0,1]$ as well as 
\[
 \alpha[x]=\int_0^1 A(s)\textup dx(s) \quad \text{and} \quad \beta[x]=\int_0^1 B(s)\textup dx(s) \qquad \text{for $x \in CBV[0,1]$}. \qedhere
\]
\end{proof}

Similar approach to the above one can also be used to study the existence of solutions to BVPs with non-local BCs slightly different  from~\eqref{eq:bvp_bc}, namely to BVPs of the following form
\begin{align}\label{eq:bvp_webb_infante}
\begin{split}
&\hspace{2cm} x''(t)=-\lambda f(t,x(t)), \qquad t \in [0,1],\\
& x(0)=\int_0^1 x(s)\textup dA(s), \qquad x(1)=\int_0^1 x(s)\textup dB(s),
\end{split}
\end{align}
where $f \colon [0,1] \times \mathbb R \to \mathbb R$ is continuous and $A,B \in BV[0,1]$. Let us note that since we are interested in classical twice continuously differentiable solutions of the BVP~\eqref{eq:bvp_webb_infante}, the integral BCs are well-posed.

Let us also add that such problems have been investigated via the fixed point index theory approach by, for example, Webb and Infante in~\cite{I-W2}. 

The BVP~\eqref{eq:bvp_webb_infante} is clearly equivalent with the perturbed nonlinear Hammerstein integral equation
\begin{equation}\label{eq:perturbed_Hammerstein_22}
x(t)=\int_0^1 (1-t)x(s)\textup dA(s) + \int_0^1 tx(s)\textup dB(s)+\lambda\int_0^1 k(t,s)f(s,x(s))\textup ds,\ t\in [0,1],
\end{equation}
where $k \colon [0,1] \times [0,1] \to \mathbb R$ is given by~\eqref{eq:kernel_2}, which means, as before, that each twice continuously differentiable function $x \colon [0,1] \to \mathbb R$ which satisfies the BVP~\eqref{eq:bvp_webb_infante} is a $CBV$-solution to~\eqref{eq:perturbed_Hammerstein_22}, and vice-versa. 

In the case of the BVP~\eqref{eq:bvp_webb_infante} we have the following existence result.

\begin{theorem}\label{thm:existence_BVP_2}
Let $f \colon [0,1]\times\mathbb R\to\mathbb R$ be a continuous function. Moreover, let the functions $A,B \in BV[0,1]$ be such that $\var_{[0,1]}A + \var_{[0,1]}(A-B)<1$. Then there exists a number $\lambda_0>0$ such that for any $\lambda \in \mathbb R$ satisfying $\abs{\lambda}\leq \lambda_0$ the BVP~\eqref{eq:bvp_webb_infante} has a solution.
\end{theorem}

\begin{proof}
Let the continuous functionals $\alpha$ and $\beta$ on the space $CBV[0,1]$ be given by
\[
 \alpha[x]=\int_0^1 x(s)\textup dA(s) \quad \text{and} \quad \beta[x]=\int_0^1 x(s)\textup dB(s) \qquad \text{for $x \in CBV[0,1]$}.
\]
Then, in view of~\cite{Lojasiewicz}*{Theorem~1.6.1}, we have
\[
 \norm{\alpha} = \sup_{\norm{x}_{BV}=1}\babs{\alpha[x]} \leq \var_{[0,1]}A \cdot \sup_{\norm{x}_{BV}=1} \norm{x}_{\infty} \leq \var_{[0,1]}A
\]
and $\norm{\alpha-\beta}\leq \var_{[0,1]}(A-B)$.

To end the proof it suffices to apply Corollary~\ref{cor:hammerstein_ver_2} with $v(t)=1-t$ and $w(t)=t$ as well as $m(s)=1$, $\phi(t)=1$ and 
\[
 \psi(r):=\sup\bset{\abs{f(t,u)}}{\text{$t\in[0,1]$ and $u\in[-r,r]$}} \qquad \text{for $r\geq 0$}
\]
(cf. the proof of Theorem~\ref{thm:bvp_nlc}).
\end{proof}

Finally, we will devote the last part of this Section to illustrating the above existence results by two examples.

\begin{example}\label{ex:111}
Let us consider the BVP
\begin{align}
&\hspace{1cm} x''(t)=-\lambda f(t,x(t)), \qquad t \in [0,1], \label{eq:bvp2i}\\[1.5mm]
& x(0)=\frac{1}{5}x(a) + \frac{1}{5}x(c), \qquad x(1)=\frac{1}{5}x(b)+\frac{1}{5}x(c), \label{eq:bvp2ii}
\end{align}
where $0<a<b<c<1$ and $f \colon [0,1] \times \mathbb R \to \mathbb R$ is a continuous function. It is easy to see that the BVP~\eqref{eq:bvp2i}--\eqref{eq:bvp2ii} is equivalent with the BVP~\eqref{eq:bvp_webb_infante} with $A=\frac{1}{5}\chi_{[a,1]}+\frac{1}{5}\chi_{[c,1]}$ and $B=\frac{1}{5}\chi_{[b,1]} + \frac{1}{5}\chi_{[c,1]}$.

Since $\var_{[0,1]}A + \var_{[0,1]}(A-B)=\frac{4}{5}<1$, from Theorem~\ref{thm:existence_BVP_2} it follows that there is $\lambda_0>0$ such that for every $\lambda \in \mathbb R$ with $\abs{\lambda} \leq \lambda_0$ the BVP~\eqref{eq:bvp2i}--\eqref{eq:bvp2ii} admits at least one solution. Of course, if the function $t \mapsto f(t,0)$ is not identically equal to zero on $[0,1]$, the solution is non-zero.

On the other hand, transforming the BVP~\eqref{eq:bvp2i}--\eqref{eq:bvp2ii} into an equivalent BVP with integral conditions of the form 
\begin{equation}\label{eq:nlc_22}
 x(0)=\int_0^1 \widehat{A}(s)\textup dx(s) \qquad \text{and} \qquad x(1)=\int_0^1 \widehat{B}(s)\textup dx(s)
\end{equation}
with some functions $\widehat{A}, \widehat{B} \in \widehat{\Omega}[0,1]$ (and then applying Theorem~\ref{thm:bvp_nlc}), in general, seems to be a more difficult task, since the functionals $\alpha[x]=\frac{1}{5}x(a) + \frac{1}{5}x(c)$ and $\beta[x]=\frac{1}{5}x(b)+\frac{1}{5}x(c)$ and the functionals generated by the right-hand sides of the BCs~\eqref{eq:nlc_22} have different properties; for example, $\alpha[e]=\beta[e]=\frac{2}{5}>0$, whereas\footnote{Let us recall that by $e$ we denote the constant function given by $e(t)=1$ for $t \in [0,1]$.}
\[
 \int_0^1 \widehat{A}(s)\textup de(s)=\int_0^1 \widehat{B}(s)\textup de(s)=0. 
\]

However, in some cases it can be done. For example, it can be checked that the BVP~\eqref{eq:bvp2i}--~\eqref{eq:bvp2ii} where the function $f$ is given by $f(t,u)=0$ for $(t,u) \in [0,1] \times \mathbb R$ and the following BVP
\begin{align*}
&\hspace{1cm} x''(t)=0, \qquad t \in [0,1],\\[1.5mm]
& x(0)=\int_0^1 \textup dx(s), \qquad x(1)=\int_0^1 \textup dx(s), 
\end{align*}
have only zero solution. This, in particular, means that those BVPs are equivalent.
\end{example}

Now, let us pass to the second example.

\begin{example}\label{ex:3}
Let us consider the BVP
\begin{align}
&\hspace{1cm} x''(t)=-\lambda f(t,x(t)), \qquad t \in [0,1], \label{eq:bvp1i}\\
& x(0)=2x(a) - 2x(c), \qquad x(1)=2x(b)-2x(c), \label{eq:bvp1ii}
\end{align}
where $0<a<b<c<1$ and $f \colon [0,1] \times \mathbb R \to \mathbb R$ is a continuous function. The BVP~\eqref{eq:bvp1i}--\eqref{eq:bvp1ii} is equivalent with~\eqref{eq:bvp}--\eqref{eq:bvp_bc}, where $A=2\chi_{[c,1]}-2\chi_{[a,1]}$ and $B=2\chi_{[c,1]} - 2\chi_{[b,1]}$ (cf. Example~\ref{ex:111} and~\cite{Lojasiewicz}*{Theorem~1.6.7}). Of course, $A,B \in \widehat{\Omega}[0,1]$, and moreover
\[
 \bgabs{\int_0^1[A(s)-B(s)]\textup ds}=2\abs{a-b}.
\]
So if, for example, $a=\frac{1}{5}$, $b=\frac{3}{5}$ and $c=\frac{4}{5}$, then $2\abs{a-b}=\frac{4}{5}<1$, and hence from Theorem~\ref{thm:bvp_nlc} it follows that there exists $\lambda_0>0$ such that for every $\lambda \in \mathbb R$ with $\abs{\lambda} \leq \lambda_0$ the BVP~\eqref{eq:bvp1i}--\eqref{eq:bvp1ii} has at least one solution. 

Now, we would like to show that in general it is not possible to transform the BVP~\eqref{eq:bvp1i}--\eqref{eq:bvp1ii} into an equivalent BVP of the form~\eqref{eq:bvp_webb_infante} to which one could apply Theorem~\ref{thm:existence_BVP_2}.

Suppose that the function $f$ is given by $f(t,u)=2$ for $(t,u) \in [0,1] \times \mathbb R$ and that the BVP~\eqref{eq:bvp1i}--\eqref{eq:bvp1ii} can be equivalently rewritten as the following BVP with non-local integral BCs 
\begin{align}
&\hspace{1cm} x''(t)=-2\lambda, \qquad t \in [0,1], \label{eq:bvp1i2}\\
&  x(0)=\int_0^1 x(s)\textup d\widehat{A}(s), \qquad x(1)=\int_0^1 x(s)\textup d\widehat{B}(s), \label{eq:bvp1ii2}
\end{align}
for some $\widehat{A}, \widehat{B} \in BV[0,1]$. Let us note that one of such equivalent reformulations of the BVP~\eqref{eq:bvp1i}--\eqref{eq:bvp1ii} can be obtained using the functions $\widehat{A}=2\chi_{[a,1]}-2\chi_{[c,1]}$ and $\widehat{B}=2\chi_{[b,1]}-2\chi_{[c,1]}$.

It can be checked that for every $\lambda \in \mathbb R$ the BVP~\eqref{eq:bvp1i}--\eqref{eq:bvp1ii} with the above-defined function $f$ and $a=\frac{1}{5}$, $b=\frac{3}{5}$, $c=\frac{4}{5}$ has a unique solution $x_{\lambda} \colon [0,1] \to \mathbb R$ given by $x_{\lambda}(t)=-\lambda t^2 + \frac{9}{5}\lambda t - \frac{24}{25}\lambda$. Since the BVPs~\eqref{eq:bvp1i}--\eqref{eq:bvp1ii} and~\eqref{eq:bvp1i2}--\eqref{eq:bvp1ii2} are equivalent, we infer that for every $\lambda \in \mathbb R$ we have
\[
 \frac{24}{25}\abs{\lambda}=\abs{x_{\lambda}(0)}=\bgabs{\int_0^1 x_{\lambda}(s)\textup d\widehat{A}(s)}\leq \norm{x_{\lambda}}_{\infty} \cdot \var_{[0,1]}\widehat{A} = \frac{24}{25}\abs{\lambda}\cdot\var_{[0,1]}\widehat{A}.
\]
In particular, $\var_{[0,1]} \widehat{A} \geq 1$, which shows that $\var_{[0,1]} \widehat{A} + \var_{[0,1]}(\widehat{A}-\widehat{B})\geq 1$ and means that the assumptions of Theorem~\ref{thm:existence_BVP_2} cannot be satisfied.
\end{example}

\begin{remark}
In connection with Example~\ref{ex:3} a natural question arises whether instead of transforming the BVP~\eqref{eq:bvp1i}--\eqref{eq:bvp1ii} into an equivalent BVP of the form~\eqref{eq:bvp_webb_infante} and trying to apply Theorem~\ref{thm:existence_BVP_2}, it would not be better to use directly Corollary~\ref{cor:hammerstein_ver_2} with the functionals $\alpha,\beta \in CBV^*[0,1]$ given by $\alpha[x]=2x(a) - 2x(c)$ and $\beta[x]=2x(b) - 2x(c)$.

It turns out that (in general) it would be not, since, as we will show below, the norm of the functional $\alpha$ is at least $2$, and thus the condition~\ref{b6} cannot be satisfied. 

Let $x \in CBV[0,1]$ be a continuous piecewise linear function whose graph is a polygonal line spanned by the points $(0,0)$, $(a,0)$, $(c,1)$ and $(1,1)$. Then $\norm{x}_{BV}=1$ and $\babs{\alpha[x]} = 2$, which proves that $\norm{\alpha}\geq 2$.
\end{remark}

\begin{remark}\label{rem:BV_etc}
The existence of (positive) solutions (and their multiplicity) to BVPs with non-local BCs of the form~\eqref{eq:bvp_webb_infante} was also studied, for example, in~\cites{I2009, W1, I-W2, I-W3}. The main tool used in those papers was the fixed point index theory in the space of continuous functions, and hence additional assumptions on the nonlinearity $f$, such as sub- or superlinearity, were required. 

For completeness, let us also add that if we use the functions $\widehat{A}=2\chi_{[a,1]}-2\chi_{[c,1]}$ and $\widehat{B}=2\chi_{[b,1]}-2\chi_{[c,1]}$ to transform the BVP~\eqref{eq:bvp1i}--\eqref{eq:bvp1ii} from Example~\ref{ex:3} into an equivalent BVP with non-local integral conditions, and additionally we assume that $a=\frac{1}{5}$, $b=\frac{3}{5}$, $c=\frac{4}{5}$, then we cannot apply the existence result from~\cite{I-W2}, since the assumption~$(C_6)$ from the aforementioned article is not satisfied.
\end{remark}

\section*{Acknowledgments}
The authors would like to thank the anonymous referee for his/her suggestions.
This paper was partially written during the visit of G. Infante to the
Optimization and Control Theory Department, Adam Mickiewicz University, Pozna\'n, Poland.
G. Infante is grateful to the people of the aforementioned Department for their kind and warm hospitality.  G. Infante was partially supported by G.N.A.M.P.A. - INdAM (Italy).

\end{document}